\documentclass[11pt,a4paper]{amsart}
\usepackage{amssymb,amsfonts,euscript,amscd,amsmath,amsthm,amsopn}
\usepackage{color, graphicx, bbm, url}
\usepackage[usenames,dvipsnames]{xcolor}
\usepackage[utf8]{inputenc}
\usepackage[position=top]{subfig}
\usepackage{standalone}
\usepackage{xspace}
\usepackage{booktabs,tabularx}
\usepackage[dvipsnames]{xcolor}
\usepackage{listings}
 
\definecolor{codegreen}{rgb}{0,0.6,0}
\definecolor{codepurple}{rgb}{0.58,0,0.82}

\lstdefinestyle{mystyle}{
language={},
basicstyle=\footnotesize\ttfamily,
xleftmargin=2em,
xrightmargin=2em,
columns=fullflexible,
keepspaces=true,         
classoffset=2,
morekeywords={minors, genericMatrix, matrix, ideal, eliminate, map, degrees, random},
keywordstyle={\color{blue}\bfseries},
classoffset=3,
morekeywords={for,from, to, do},
keywordstyle={\color{codepurple}\bfseries},
classoffset=4,
morekeywords={Height},
keywordstyle={\color{Emerald}\bfseries},
classoffset=5,
morekeywords={QQ, ZZ},
keywordstyle={\color{codegreen}\bfseries},
stepnumber=1,
numbers=left,
captionpos=b,
showspaces=false,
showstringspaces=false,
morestring=[b]",
frame=single,
}
 
\title{Rigid Multiview Varieties}
\author{Michael Joswig \and Joe Kileel \and Bernd Sturmfels \and Andr\'e Wagner}

\def\cocoa{{\hbox{\rm C\kern-.13em o\kern-.07em C\kern-.13em o\kern-.15em A}}}

\newcommand{\QQ}{\mathbb{Q}}

\newcommand{\RR}{\mathbb{R}}
\newcommand{\CC}{\mathbb{C}}
\newcommand{\ZZ}{\mathbb{Z}}

\newcommand{\PP}{\mathbb{P}}

\DeclareMathOperator*{\rk}{rk} 
 
\DeclareMathOperator{\Sym}{Sym}
\DeclareMathOperator{\GL}{GL}

\newcommand\macaulay{\texttt{Macaulay2}\xspace}
\newcommand\gfan{\texttt{Gfan}\xspace}

\newcommand{\inD}[1][\relax]{\def\argone{#1}\def\temprelax{\relax}
  \ifx\argone\temprelax\right.\else\,\middle|#1\right.{}\fi}

\theoremstyle{plain}
\newtheorem{thm}{Theorem}
\newtheorem{lem}[thm]{Lemma}
\newtheorem{prop}[thm]{Proposition}
\newtheorem{cor}[thm]{Corollary}

\newtheorem{conj}[thm]{Conjecture}

\theoremstyle{definition}
\newtheorem{exmp}[thm]{Example}

\theoremstyle{remark}
\newtheorem{rem}[thm]{Remark}

\begin{document} 
 
\begin{abstract} \noindent
The multiview variety from 
computer vision is generalized to 
images by $n$ cameras of points
linked by a distance constraint.
The resulting five-dimensional
variety  lives in a product of
$2n$ projective planes. We determine defining 
polynomial equations,
and we explore generalizations of this variety
to scenarios of interest in applications.
\end{abstract}


\maketitle

\section{Introduction} \label{sec:one}

The emerging field of Algebraic Vision is concerned with interactions between computer vision and algebraic geometry.
A central role in this endeavor is played by projective varieties that arise in multiview geometry~\cite{HZ}.

The set-up is as follows:
 A \emph{camera} is a linear map from the three-dimensional projective space $\PP^3$ to the
 projective plane $\PP^2$, both over $\RR$.
 We represent $n$ cameras by matrices
$A_1,A_2,\ldots,A_n\in \RR^{3\times4}$ of rank~$3$.
The kernel of $A_j$ is the {\em focal point} $f_j \in \PP^3$.
Each image point 
$u_j \in \PP^2$ of camera $A_j$ has a line through $f_j$ as its fiber
in $\PP^3$. This is the \emph{back-projected line}.

We assume throughout
that the focal points of the $n$ cameras are 
in  {\em general position},
i.e.~all distinct, no three on a line, and no four on a plane.
Let $\beta_{jk}$ denote the line in $\PP^3$ spanned by
the focal points $f_j$ and $f_k$.
This is the \emph{baseline} of the camera pair $A_j, A_k$.
   The image of the
focal point $f_j$ in the image plane $\PP^2$ of the camera $A_k$ is the \emph{epipole} $e_{k\leftarrow j}$.  Note that the
baseline $\beta_{jk}$ is the back-projected line of $e_{k\leftarrow j}$ with respect to $A_j$ and also the back-projected line of $e_{j\leftarrow k}$ with respect to $A_k$.  See Figure \ref{fig:2view} for a sketch.

\begin{figure}[th]
  \includegraphics[width=\textwidth]{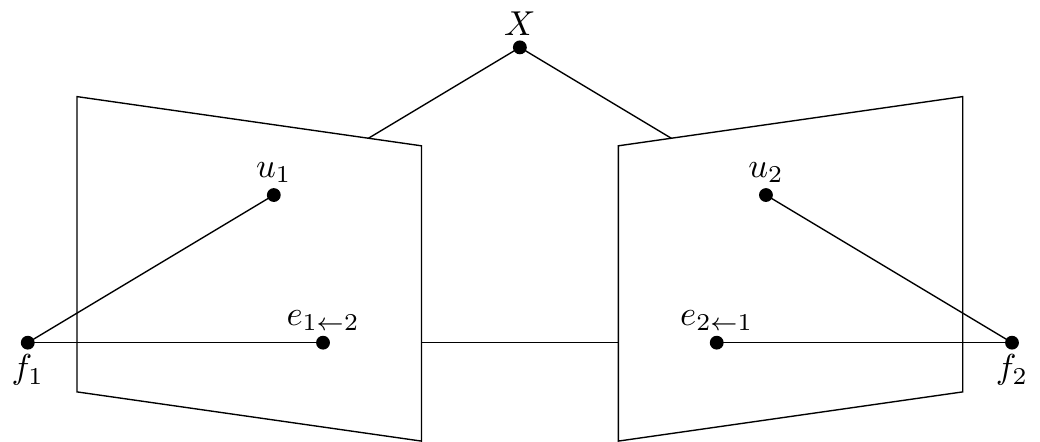}%
   \caption{Two-view geometry}
  \label{fig:2view}
\end{figure}

Fix a point $X$ in $\PP^3$ which is not on the baseline $\beta_{jk}$,
and let $u_j$ and $u_k$ be the images of $X$ under $A_j$ and $A_k$.
 Since $X$ is not on the baseline, neither image
point is the epipole for the other camera.  The two back-projected lines of $u_j$ and $u_k$ meet in a unique
point, which is $X$.  This process of reconstructing $X$ from two images $u_j$
and $u_k$ is called \emph{triangulation} \cite[\S9.1]{HZ}.

The triangulation procedure amounts to solving the linear equations
\begin{equation}\label{eq:triangulation}
\quad  B^{jk} \begin{bmatrix}\,X\\-\lambda_j\\-\lambda_k\end{bmatrix} \, = \, 0
\qquad \hbox{where} \quad
B^{jk} \,=\, \begin{bmatrix} A_j & u_j & 0 \\ A_k & 0 & u_k\end{bmatrix} \,\in \,\RR^{6 \times 6}.
\end{equation}
For general data we have ${\rm rank}(B^{jk}) = 
{\rm rank}(B^{jk}_1) = \! \cdots \! = {\rm rank}(B^{jk}_6)  = 5$,
where $B^{jk}_i$ is obtained  from $B^{jk}$
by deleting the $i$th row.
{\em Cramer's Rule}  can be used to recover $X$.
Let $\wedge_5 B_i^{jk} \in \RR^6$ be the column vector formed by the
signed maximal minors of $B^{jk}_i$.  Write 
$\widetilde{\wedge}_5 B_i^{jk} \in \RR^4$ for the first four coordinates
of $\wedge_5 B_i^{jk}$. These are bilinear functions of
$u_j$ and $u_k$. They yield
\begin{equation}\label{eq:X}
  X \, = \, \widetilde{\wedge}_5 B_1^{jk} \, = \, \widetilde{\wedge}_5 B_2 ^{jk} \, = \, \cdots \, = \,
  \widetilde{\wedge}_5 B_6^{jk}.
\end{equation}
We note that, in most practical applications,
the data $u_1,\ldots,u_n$ will be noisy,
in which case  triangulation 
requires techniques from optimization \cite{AAT}.

The \emph{multiview variety} $V_A$ of the camera configuration $A = (A_1,\dots,A_n)$
was defined in \cite{AST} as the closure of the image of the rational map 
\begin{equation}
  \label{eq:multiMap}
  \begin{matrix}
    \phi_{A}:&  \PP^3  & \dashrightarrow &  \PP^2 \times \PP^2 \times \cdots \times \PP^2, \\
&    X & \mapsto&  (A_1X, A_2X, \ldots, A_nX).
  \end{matrix}
  \end{equation}
The points $(u_1,u_2,\ldots,u_n) \in V_A$ are the 
consistent views in $n$ cameras.
The prime ideal $I_A$ of $V_A$ was determined in \cite[Corollary 2.7]{AST}.
It is generated by the $\binom{n}{2}$ bilinear polynomials
${\rm det}(B^{jk})$ plus $\binom{n}{3}$ further trilinear polynomials.
See \cite{Li} for the natural generalization  
of this variety to higher dimensions.

The analysis in \cite{AST} was restricted to a single world point $X \in \PP^3$.
In this paper we study the case of two world points $X,Y \in \PP^3$
that are linked by a distance constraint.
 Consider the hypersurface $V(Q)$ in $\PP^3 \times \PP^3$  defined by
 \begin{equation}
 \label{eq:biquadratic} \quad
Q \,= \,(X_0Y_3 - Y_0X_3)^2 + (X_1Y_3 - Y_1X_3)^2 + (X_2Y_3 - Y_2X_3)^2 - X_3^2Y_3^2.
\end{equation}
The affine variety $V_{\RR}(Q) \cap \{X_3 {=} Y_3 {=}1\}$
in $\RR^3 \times \RR^3$ consists of
pairs of points whose Euclidean distance is $1$.
The \emph{rigid multiview map} is the rational map
\begin{equation}
\label{eq:rigidMap}
\begin{matrix}
\psi_{A}: & V(Q) & \hookrightarrow & \PP^3 \times \PP^3 & \dashrightarrow & 
(\PP^2)^n \times (\PP^2)^n, \\
& (X,Y) & & & \mapsto & \bigl((A_1X, \ldots A_nX), (A_1Y, \ldots A_nY)\bigr)  .
\end{matrix}
\end{equation}
The \emph{rigid multiview variety} is the image of this map.  This is a $5$-dimensional subvariety of $(\PP^2)^{2n}$. Its
multihomogeneous prime ideal $J_A$ lives in the polynomial ring $ \RR[u,v] = \RR[u_{i0},u_{i1},u_{i2},\,
v_{i0},v_{i1}, v_{i2}: \, i=1, \ldots, n]$, where $(u_{i0}{:}u_{i1}{:}u_{i2})$ 
and $(v_{i0}{:}v_{i1}{:}v_{i2})$ are coordinates for the $i$th factor
 $\PP^2$ on the left respectively right in 
$(\PP^2)^n \times (\PP^2)^n$.
Our aim is to determine the ideal $J_A$. 
Knowing generators of $J_A$ has the potential
of being useful for designing optimization tools
as in \cite{AAT} for triangulation in the presence of distance constraints.

The choice of world and image coordinates 
for the camera configuration $A = (A_1,\dots,A_n)$ gives our
problem the following group symmetries.
Let $N$ be an element of the \emph{Euclidean group of motions} $\text{SE}(3,\RR) $, which is generated by rotations and translations.
We may multiply the camera configuration on the right by $N$ to obtain 
$AN = (A_{1}N,\dots,A_{n}N)$.  Then $J_A = J_{AN}$
since $V(Q)$ is invariant under $\text{SE}(3,\RR)$.
For $M_{1}, \ldots, M_{n} \in \text{GL}(3,\RR)$, we may multiply $A$ on the left to obtain
$A' = (M_{1}A, \ldots, M_{n}A)$.  Then $J_{A'} = (M_1 \otimes \ldots \otimes M_n) J_A$.

This paper is organized as follows.  In Section 2 we present the explicit computation of the rigid multiview ideal for
$n=2,3,4$. Our main result, to be stated and proved in Section 3, is a system of equations that cuts out the rigid
multiview variety $V(J_A)$ for any $n$.  Section 4 is devoted to generalizations.  The general idea is to replace $V(Q)$
by arbitrary subvarieties of $(\PP^3)^m$ that represent polynomial constraints on $m \geq 2 $ world points.  We focus on
scenarios that are of interest in applications to computer vision.

Our results in Propositions \ref{prop:n2}, \ref{prop:n3}, \ref{prop:n4} and Corollary \ref{cor:mingensH_A} are proved by computations with \macaulay \cite{M2}; for details see Appendix \ref{compu}.
Following standard practice in computational algebraic geometry, we carry out the computation on many samples in a Zariski dense set of parameters, and then conclude that it holds generically.

\section{Two, Three and Four Cameras}\label{sec:234}

In this section we offer a detailed case study of the rigid multiview variety 
when the number $n$ of cameras is small.
We begin with the case~$n=2$.  The prime ideal $J_A$ lives in 
the polynomial ring $\RR[u,v]$ in $12$ variables. This is
the homogeneous coordinate ring of $(\PP^2)^4$, so it is naturally $\ZZ^4$-graded.
The variables $u_{10},u_{11},u_{12}$ have degree $(1,0,0,0)$,
the variables $u_{20},u_{21},u_{22}$ have degree $(0,1,0,0)$,
the variables $v_{10},v_{11},v_{12}$ have degree $(0,0,1,0)$, and
the variables $v_{20},v_{21},v_{22}$ have degree $(0,0,0,1)$.
Our ideal $J_A$ is $\ZZ^4$-homogeneous.

Throughout this section we shall assume that the camera configuration $A$ 
is \emph{generic} in the sense of algebraic geometry. This means that $A$
lies in the complement of a certain (unknown) proper algebraic subvariety
in the affine space of all $n$-tuples of  $3 \times 4$-matrices.
All our results in Section 2 were obtained by symbolic computations with
sufficiently many random choices of $A$ (see Appendix \ref{compu} for details).
Such choices of camera matrices are generic. They will be attained with with probability~$1$.

\begin{prop} \label{prop:n2}
  For $n = 2$, the rigid multiview ideal $J_A$ is minimally generated
  by eleven $\,\ZZ^4$-homogeneous polynomials in twelve variables, one
  of degree $(1,1,0,0)$, one of degree $(0,0,1,1)$, and nine of degree
  $(2,2,2,2)$.
\end{prop}

We prove this result by sufficiently many random computations with \macaulay.
A slightly simplified version of the code is shown in Listing \ref{algo1} in Appendix \ref{compu}.

Let us look at the result in more detail.
The first two bilinear generators are the familiar $6 \times 6$-determinants
\begin{equation}
\label{eq:twobilinear}
{\rm det} \begin{bmatrix} A_1 & u_1 & 0 \\ A_2 & 0 & u_2 \end{bmatrix} 
\qquad \hbox{and} \qquad
{\rm det} \begin{bmatrix} A_1 & v_1 & 0 \\ A_2 & 0 & v_2 \end{bmatrix} .
\end{equation}
These cut out two copies of the multiview threefold $V_A \subset (\PP^2)^2$, in separate variables, for $X \mapsto u=
(u_1,u_2)$ and $Y \mapsto v= (v_1,v_2)$.  If we write the two bilinear forms in (\ref{eq:twobilinear}) as $u_1^\top F u_2$
and $v_1^\top F v_2$ then $F$ is a real $3 \times 3$-matrix of rank~$2$, known as the {\em fundamental matrix} \cite[\S
9]{HZ} of the camera pair $(A_1,A_2)$.

The rigid multiview variety $V(J_A)$ is a divisor in 
$V_A \times V_A \subset (\PP^2)^2 \times (\PP^2)^2$.  The nine
octics that cut out this divisor can be understood as follows.  We 
write $B$ and $C$ for
the  $6 \times 6$-matrices in
(\ref{eq:twobilinear}), and $B_i$ and $C_i$ for the matrices obtained 
by deleting their $i$th rows. The kernels of these
$5 \times 6$-matrices are represented, via Cramer's Rule, 
by $\wedge_5 B_i$ and $\wedge_5 C_i$.  We write
$\widetilde{\wedge}_5 B_i$ and $\widetilde{\wedge}_5 C_i$ for the vectors 
given by their first four entries. As in
(\ref{eq:X}), these represent the two world points $X$
and $Y$ in $\PP^3$.  Their coordinates are bilinear forms in
$(u_1,u_2)$ or $(v_1,v_2)$, where each coefficient is a $3 \times 3$-minor of $\begin{small} \begin{bmatrix} A_1 \\
    A_2 \end{bmatrix} \end{small}$.  For instance, writing $a^{jk}_i$ for the $(j,k)$ entry of $A_i$, the first
coordinate of $\widetilde{\wedge}_5 B_1$~is
\[
\begin{small}
\begin{matrix}
-(a_1^{32} a_2^{23} a_2^{34}-a_1^{32} a_2^{24} a_2^{33}-a_1^{33} a_2^{22}
a_2^{34} + a_1^{33} a_2^{24} a_2^{32}+a_1^{34} a_2^{22} a_2^{33}-
a_1^{34} a_2^{23} a_2^{32})\, u_{11} u_{20} \\ 
+ (a_1^{32} a_2^{13} a_2^{34}-a_1^{32} a_2^{14} a_2^{33}-a_1^{33} a_2^{12}
a_2^{34}+a_1^{33} a_2^{14} a_2^{32}+a_1^{34} a_2^{12} a_2^{33}
- a_1^{34} a_2^{13} a_2^{32}) \,u_{11} u_{21} \\
- (a_1^{32} a_2^{13} a_2^{24}-a_1^{32} a_2^{14} a_2^{23}
-a_1^{33} a_2^{12} a_2^{24}+a_1^{33} a_2^{14} a_2^{22}
+a_1^{34} a_2^{12} a_2^{23}-a_1^{34} a_2^{13} a_2^{22}) \,u_{11} u_{22} \\
+ (a_1^{22} a_2^{23} a_2^{34} - a_1^{22} a_2^{24} a_2^{33}
-a_1^{23} a_2^{22} a_2^{34} + a_1^{23} a_2^{24} a_2^{32} + a_1^{24} a_2^{22}
 a_2^{33}- a_1^{24} a_2^{23} a_2^{32}) \,u_{12} u_{20} \\
 -( a_1^{22} a_2^{13} a_2^{34} - a_1^{22} a_2^{14} a_2^{33}-a_1^{23}
 a_2^{12} a_2^{34} + a_1^{23} a_2^{14} a_2^{32} + a_1^{24} a_2^{12} a_2^{33}
- a_1^{24} a_2^{13} a_2^{32})\, u_{12} u_{21} \\ \,
+(a_1^{22} a_2^{13} a_2^{24}
-a_1^{22} a_2^{14} a_2^{23} -a_1^{23} a_2^{12} a_2^{24}
+a_1^{23} a_2^{14} a_2^{22}+a_1^{24} a_2^{12} a_2^{23}
-a_1^{24} a_2^{13} a_2^{22})\, u_{12} u_{22}.
\end{matrix}
\end{small}
\]

Recall that the two world points in $\PP^3$ are linked by a distance
constraint~(\ref{eq:biquadratic}), expressed as a biquadratic
polynomial $Q$. We set $Q(X,Y)=T(X,X,Y,Y)$, where $T( \bullet
,\bullet,\bullet,\bullet)$ is a quadrilinear form.  We regard $T$ as a
tensor of order~$4$.  It lives in the subspace $\, \Sym_2 ( \RR^4)
\otimes \Sym_2 ( \RR^4) \simeq \RR^{100}\,$ of $\,(\RR^4)^{\otimes 4}
\simeq \RR^{256}$.  Here $\Sym_k(\,\cdot\,)$ denotes the space of
symmetric tensors of order $k$.

We now substitute our Cramer's Rule formulas for $X$ and $Y$
into the quadrilinear form $T$.  For any choice of indices
$1 {\leq} i {\leq} j {\leq} 6$ and $1 {\leq} k {\leq} l {\leq} 6$, 
  \begin{equation}
  \label{eq:zweizweizweizwei}
  T\bigl(\,\widetilde{\wedge}_5 B_i\,,\,
    \widetilde{\wedge}_5 B_j\,,\,
    \widetilde{\wedge}_5 C_k \,,\,
    \widetilde{\wedge}_5 C_l\,\bigr) 
    \end{equation}
     is a multihomogeneous 
    polynomial in $(u_1,u_2,v_1,v_2)$ of
    degree $(2,2,2,2)$.
    This polynomial lies in $J_A$
    but not in the ideal  $I_A(u) + I_A(v)$
    of $V_A \times V_A$, so it can 
    serve as one of the nine minimal generators
    described in Proposition~\ref{prop:n2}.

The number of distinct polynomials appearing in (\ref{eq:zweizweizweizwei})
equals $\binom{7}{2}^{2}=441$. 
A computation verifies that these polynomials span a real vector space 
of dimension $126$. The image of that vector space modulo
the degree $(2,2,2,2)$ component of the ideal $I_A(u)+I_A(v)$
has dimension $9$.

We record three more features of the rigid multiview
with $n=2$ cameras. The first is the {\em multidegree} \cite[\S 8.5]{MS},
or, equivalently, the
cohomology class of $\,V(J_A)\,$ in
$\,H^*\bigl((\PP^2)^4 ,\ZZ\bigr)  = 
\ZZ[u_1,u_2,v_1,v_2]/\langle u_1^3,u_2^3,v_1^3,v_2^3\rangle$.
It equals
\[ \begin{matrix} 
\,\, 2u_1^2v_1
+2u_1u_2v_1
+2u_2^2v_1
+2u_1^2v_2
+2u_1u_2v_2
+2u_2^2v_2 \\
+2u_1v_1^2
+2u_1v_1v_2
+2u_1v_2^2
+2u_2v_1^2
+2u_2v_1v_2
+2u_2v_2^2.
\end{matrix}
\]
This is found with the built-in command {\tt multidegree} in \macaulay.

The second is the table of the Betti numbers of the minimal free resolution of~$J_A$ in the format of \macaulay \cite{M2}. In that format, the columns
correspond to the syzygy modules, while rows denote the degrees.  
For $n=2$ we obtain
\begin{small}
\begin{verbatim}
                           0  1  2  3 4 5
                    total: 1 11 25 22 8 1
                        0: 1  .  .  . . .
                        1: .  2  .  . . .
                        2: .  .  1  . . .
                        7: .  9 24 22 8 1
\end{verbatim}
\end{small}
The column labeled {\tt 1} lists the minimal generators from Proposition~\ref{prop:n2}.
Since the codimension of $V(J_A)$ is $3$, the table shows that $J_A$ is not Cohen-Macaulay. The unique $5$th syzygy has
degree $(3,3,3,3) $ in the $\ZZ^4$-grading.

The third point is an explicit choice for the nine generators of degree $(2,2,2,2) $ in Proposition \ref{prop:n2}.
Namely, we take $i=j \leq 3$ and $k=l \leq 3$ in (\ref{eq:zweizweizweizwei}).
The following corollary is also found by computation:  
  
\begin{cor}\label{cor:mingensH_A}
The rigid multiview ideal $J_A$ for $n=2$ is 
generated by $\,I_A(u) + I_A(v)$ together with the nine polynomials
  $Q\bigl(\widetilde{\wedge}_5 B_{i},\widetilde{\wedge}_5 C_{k}\bigr )$ for $\,1\leq i,k\leq 3$.
\end{cor}

We next come to the case of three cameras:

\begin{prop}
\label{prop:n3}
For $n = 3$, the rigid multiview ideal $J_A$ is minimally generated by $\,177$ polynomials in $18$ variables.
Its Betti table is given in Table~\ref{tab:n3}.
\end{prop}

\begin{table}
\begin{small}
\begin{verbatim}
          0   1    2    3     4     5     6    7    8   9  10 11
   total: 1 177 1432 5128 10584 13951 12315 7410 3018 801 126  9
       0: 1   .    .    .     .     .     .    .    .   .   .  .
       1: .   6    .    .     .     .     .    .    .   .   .  .
       2: .   2   21    6     .     .     .    .    .   .   .  .
       3: .   .    6   36    18     .     .    .    .   .   .  .
       4: .   .    1   12    42    36     9    .    .   .   .  .
       5: .   1    .    .     .     .     .    .    .   .   .  .
       6: .  24  108  166   120    42     6    .    .   .   .  .
       7: . 144 1296 4908 10404 13873 12300 7410 3018 801 126  9
\end{verbatim}
\end{small}
\smallskip
\caption{\label{tab:n3} Betti numbers for the rigid multiview ideal with $n=3$.}
\end{table}

Proposition \ref{prop:n3} is proved by computation.
The $177$ generators occur in eight symmetry classes
of multidegrees. Their numbers in these classes are
\[
\begin{matrix}
 (110000) : 1 &&    (220111):3 && (220220):9 &&    (211211):1 \\ 
(111000) : 1 & &    (211111):1 && (220211):3 &&  (111111):1 
\end{matrix}
\]
For instance, there are nine generators in degree $(2,2,0,2,2,0)$, arising 
from Proposition~\ref{prop:n2} for the first two cameras. Using various pairs
among the three cameras when forming the matrices $B_i, B_j, C_k$ and $C_l$
in (\ref{eq:zweizweizweizwei}), we can construct the generators  
of degree classes $(2,2,0,2,1,1)$ and $(2,1,1,2,1,1)$.

Table \ref{tab:n3} shows the Betti table for $J_A$ in \macaulay format.
The first two entries ({\tt 6} and {\tt 2}) in the 
{\tt 1}-column refer to the eight minimal generators of $I_A(u) + I_A(v)$. 
These are six bilinear forms,
representing the three fundamental matrices, and two trilinear forms, representing the {\em trifocal tensor} of the
three cameras (cf.~\cite{AO}, \cite[\S 15]{HZ}).  The entry 
{\tt 1} in row {\tt 5} of column~{\tt 1} marks the unique sextic
generator of $J_A$, which has $\ZZ^6$-degree $(1,1,1,1,1,1)$.

For the case of four cameras we obtain the following result.

\begin{prop}
\label{prop:n4}
For $n = 4$, the rigid multiview ideal $J_A$ is minimally
generated by $1176$ polynomials in $24$ variables. 
All of them are induced from $n=3$.
Up to symmetry, the degrees of the generators in the $\ZZ^8$-grading are
\[
\begin{matrix}
 (11000000) : 1 &&    (22001110):3 && (22002200):9 &&    (21102110):1 \\ 
(11100000) : 1 & &    (21101110):1 && (22002110):3 &&  (11101110):1 
\end{matrix}
\]
\end{prop}

We next give a brief explanation of how the rigid multiview ideals
$J_A$ were computed with \macaulay \cite{M2}. For the purpose of
efficiency, we introduce projective coordinates for the image points
and affine coordinates for the world points.  We work in the
corresponding polynomial ring
\[\QQ[u,v][X_0,X_1,X_2,Y_0,Y_1,Y_2].\]
The rigid multiview map $\psi_A$ is thus restricted to $\RR^3 \times
\RR^3$.
The prime ideal of its graph is generated by the following two classes
of polynomials:
\begin{enumerate}
\item the $2\times 2$ minors of the $3\times 2$ matrices 
\[\left [\begin{array}{c|c}
A_i\cdot (X_0,X_1,X_2,1)^\top&u_i\end{array}
\right ], \enspace \left [\begin{array}{c|c}
A_i\cdot (Y_0,Y_1,Y_2,1)^\top&v_i\end{array}
\right ],
\]
\item the dehomogenized distance constraint 
\[Q\bigl((X_0,X_1,X_2,1)^\top,(Y_0,Y_1,Y_2,1)^\top\bigr).\]
\end{enumerate}
From this ideal we eliminate the six
 world coordinates $\{X_0,X_1,X_2,Y_0,Y_1,Y_2\}$.

For a speed up, we exploit the group actions 
described in Section \ref{sec:one}.  We replace $A = (A_1, ..., A_n)$ and 
$Q = Q(X,Y)$ by $A' = (M_1A_1N, ..., M_nA_nN)$ and 
$Q' = Q(N^{-1}X, N^{-1}Y)$.  Here $M_i \in \GL_3(\RR)$ 
and $N \in \GL_4(\RR)$ are chosen so that $A'$ is sparse.
The modification to $Q$ is needed since we generally use $N \notin \text{SE}(3, \RR)$.  
The elimination above now computes the ideal
$(M_1 \otimes \ldots \otimes M_n) J_A$, and it terminates much faster.  For example,
for $n=4$, the computation took two minutes for sparse $A'$
and more than one hour for non-sparse $A$.  For $n=5$, \macaulay
ran out of memory after 18 hours of CPU time for non-sparse $A$. The complete code used in this paper can be accessed via \url{http://www3.math.tu-berlin.de/combi/dmg/data/rigidMulti/}.

One last question is whether the 
Gr\"obner basis property in \cite[\S 2]{AST}
extends to the rigid case.
This does not seem to be the case in general.
Only in Proposition~\ref{prop:n2}
can we choose minimal generators
 that form a Gr\"obner basis.

\begin{rem}\label{rem:octic} Let $n=2$. The reduced Gr\"obner basis of $J_A$
in the reverse lexicographic term order is a minimal generating set.
For a generic choice of cameras the initial ideal equals
\[
\begin{matrix}
{\rm in}(J_A) \, = \, \langle \,
 u_{10} u_{20},\,
 v_{10} v_{20},  \,
      u_{10}^2 u_{21}^2 v_{10}^2 v_{21}^2,\,
      u_{10}^2 u_{21}^2 v_{11}^2 v_{20} v_{21},\,
      u_{10}^2 u_{21}^2 v_{11}^2 v_{20}^2, \\ \phantom{da}
          u_{11}^2 u_{20}^2 v_{10}^2 v_{21}^2,\,\,
          u_{11}^2 u_{20} u_{21} v_{10}^2 v_{21}^2,      \,\,
          u_{11}^2 u_{20}^2 v_{11}^2 v_{20} v_{21},  \\ \phantom{dada} \quad
          u_{11}^2 u_{20}^2 v_{11}^2 v_{20}^2,\,
          u_{11}^2 u_{20} u_{21} v_{11}^2 v_{20} v_{21},\,
          u_{11}^2 u_{20} u_{21} v_{11}^2 v_{20}^2 \,\rangle.
\end{matrix}
\]
For special cameras the exact form of the initial ideal may change.
However, up to symmetry the degrees of the generators in the $\ZZ^4$-grading stay the same.
In general, a universal Gr\"obner basis for the rigid multiview ideal $J_A$ consists of  octics of degree $(2,2,2,2)$ plus the two quadrics  (\ref{eq:twobilinear}).
This was verified using the \gfan \cite{gfan} package in \macaulay.
Analogous statements do not hold for $n \geq 3$.
\end{rem}

\section{Equations for the Rigid Multiview Variety}

The computations presented in Section 2 suggest the following conjecture.

\begin{conj}
\label{conj:atmostdegree8} 
{\it The rigid multiview ideal $J_A$ is minimally generated
by $\,\frac{4}{9}n^6-\frac{2}{3}n^5+ \frac{1}{36}n^4
+\frac{1}{2}n^3+\frac{1}{36}n^2-\frac{1}{3}n\,$
polynomials.
These polynomials come from two triples of cameras,
and their  number per class of degrees~is}
\[
\begin{array}{rlrl}
 (110..000..) :& \,1 \cdot 2\binom{n}{2}  &   (220..111..):&\,  3 \cdot  2 \binom{n}{2}\binom{n}{3}\\   
(220..220..):&\, 9 \cdot \binom{n}{2}^2   &    (211..211..):&\, 1\cdot n^2 \binom{n-1}{2}^2 \smallskip \\   
(111..000..): &\, 1 \cdot 2 \binom{n}{3} &     (211..111..):&\,  1 {\cdot} 2n \binom{n-1}{2} \binom{n}{3}   \\
(220..211..):&\, 3 {\cdot}  2 n \binom{n}{2} \binom{n-1}{2} & (111..111..):&\, 1 \cdot \binom{n}{3}^2
\end{array}
\]
\end{conj}

\setcounter{MaxMatrixCols}{12}
\begin{table}
\begin{center}
\begin{tabular}[H]{r@{\hspace{2em}}rr@{\hspace{1.5em}}rrr@{\hspace{2em}}rrrr@{\hspsace{\textcolor{red}{4em}}}rrr}
\toprule
$n\backslash$degree & 2 & 3 & 6 & 7 & 8 & total & timing (s) \\
\midrule
2 &  2 & \phantom{0} & \phantom{0} & \phantom{0} & 9 &  1 &  $<1$ \\
3 &  6 & 2 &  1 & 24 & 144 & 177 &  14\\
4 & 12 & 8 &  16 & 240 & 900 &  1176 & 130 \\
5 & 20 & 20 &  100 & 1200 & 3600 & 4940 & 24064 \\
\bottomrule
\end{tabular}
\end{center}
\caption{\label{tab:gens} 
The known minimal generators of the rigid multiview ideals, listed by total degree, for up to five cameras.
There are no minimal generators of degrees $4$ or $5$.  
Average timings (in seconds), using the speed up described above,
are in the last column.}
\end{table}

At the moment we have a computational proof only up to $n = 5$.
Table~\ref{tab:gens} offers a summary of the corresponding numbers of generators.

Conjecture \ref{conj:atmostdegree8}  implies that $V(J_A)$ is set-theoretically
defined by the equations coming from  triples
of cameras. It turns out that,
for the set-theoretic description,  pairs of cameras suffice.
The following  is our main result:

\begin{thm} \label{thm:six} Suppose that the $n$ focal points of $A$ are in general position
in $\PP^3$.
The rigid multiview variety $V(J_A)$ is cut out as a subset of $V_A \times V_A $ by the $9
  \binom{n}{2}^2 $ octic generators of degree class $(220..220..)$. In other words, equations coming from any two pairs of
  cameras suffice set-theoretically.
\end{thm}

With notation as in the introduction, 
the relevant octic polynomials are
\[T \bigl(\,\widetilde{\wedge}_5 B_{i_1}^{j_1k_1}\,,\,\widetilde{\wedge}_5
  B_{i_2}^{j_1k_1},\widetilde{\wedge}_5 C_{i_3}^{j_2k_2} \, ,\,
  \widetilde{\wedge}_5 C_{i_4}^{j_2k_2} \,\bigr) , \]
  for all possible choices of indices.
  Let $H_A$ denote the ideal generated by these polynomials
  in $\RR[u,v]$, the polynomial ring in $6n$ variables.
  As before, we write $I_A(u) + I_A(v)$ 
  for the prime ideal  that defines the $6$-dimensional
  variety $V_A \times V_A$ in $(\PP^2)^n \times (\PP^2)^n$.
  It is generated by $2 \binom{n}{2} $ bilinear forms
  and $2 \binom{n}{3}$ trilinear forms, corresponding
  to fundamental matrices and trifocal tensors.
In light of Hilbert's Nullstellensatz,
Theorem \ref{thm:six} states that
the radical of $H_A + I_A(u) + I_A(v)$ is equal to $J_A$.
{To prove this, we need a lemma.

A point $u$ in the multiview variety $V_A\subset(\PP^2)^n$ is \emph{triangulable}
 if there exists a pair of indices $(j,k)$ such that the matrix $B^{jk}$ has rank $5$.  
 Equivalently, there exists a pair of cameras for which the unique
world point $X$ can be found by triangulation.  Algebraically,
 this means $X = \widetilde{\wedge}_5 B_{i}^{jk}$ for some
$i$.

\begin{lem}\label{l:nonTriangOn}
  All points in $V_A$ are triangulable except for the pair of epipoles, $(e_{1\leftarrow 2},e_{2\leftarrow 1})$, in the 
  case where $n=2$.  Here, the rigid multiview variety $V(J_A)$ contains the threefolds
  $V_A(u)\times (e_{1\leftarrow 2},e_{2\leftarrow 1})\,$ and $ \,(e_{1\leftarrow 2},e_{2\leftarrow 1})\times V_A(v)$.
\end{lem}

\begin{proof}
  Let us first consider the case of $n=2$ cameras.  The first claim holds because the back-projected lines of the two
  camera images $u_1$ and $u_2$ always span a plane in $\PP^3$ except when $u_1=e_{1\leftarrow 2}$ and
  $u_2=e_{2\leftarrow 1}$.  In that case both back-projected lines agree with the common baseline~$\beta_{12}$.
  Alternatively, we can check
  algebraically that the variety defined by the $5 \times 5$-minors of the matrix $B$ consists of the
  single point $ (e_{1\leftarrow 2},e_{2\leftarrow 1})$.
  
  For the second claim, fix a generic point $X$ in $\PP^3$ and consider the surface
  \begin{equation}
  \label{eq:surface} X^Q \,\, = \,\,
  \bigl\{Y \in \PP^3 \,:\, Q(X,Y)=0 \bigr\}.
  \end{equation}
Working over $\CC$, the baseline $\beta_{12}$ is either 
tangent to $X^Q$, or it meets that quadric in exactly two points.  Our
  assumption on the genericity of $X$ implies that no 
  point in the intersection $\beta_{12}\cap X^Q$ is a focal point.
  This gives
  \begin{equation}
  \label{eq:secondclaim}
  (A_1 X, A_2 X, A_1 Y_X, A_2 Y_X) \,=\, (A_1 X, A_2 X, e_{1\leftarrow 2}, e_{2\leftarrow 1}).
  \end{equation}
  The point $(A_1 X, A_2 X)$ lies in the multiview variety $V_A(u)$.  Each generic point in $V_A(u)$ has this form for
  some $X$.  Hence (\ref{eq:secondclaim}) proves the desired inclusion $V_A(u)\times (e_{1\leftarrow 2},e_{2\leftarrow
    1}) \subset V(J_A)$.  The other inclusion $ \,(e_{1\leftarrow 2},e_{2\leftarrow 1})\times V_A(v) \subset V(J_A)$
  follows by switching the roles of $u$ and $v$.

  If there are more than two cameras then for each world point $X$, due to general position of the cameras, there is a
  pair of cameras such that $X$ avoids the pair's baseline.  This shows that each point is triangulable if $n\geq 3$.
\end{proof}

\begin{proof}[Proof of Theorem \ref{thm:six}]
It follows immediately from the definition of the ideals in question  that the following inclusion 
  of varieties holds in $(\PP^2)^n \times (\PP^2)^n $:
    \[
  V(J_A) \, \subseteq \, V\bigl(I_A(u)+I_A(v)+H_A\bigr).
  \]
  We prove the reverse inclusion.
   Let $(u,v)$ be a point in the right hand side.
  
   Suppose that $u$ and $v$ are both triangulable. Then $u$ has a unique
  preimage $X$ in $\PP^3$,  determined by
  a single camera pair $\{A_{j_1},A_{k_1}\}$. Likewise, $v$ has a unique
  preimage $Y$ in $\PP^3$, also determined by
  a single camera pair $\{A_{j_2},A_{k_2}\}$.
  There exist indices $i_1,i_2 \in \{1,2,3,4,5,6\}$ such that
  \[
  X = \widetilde{\wedge}_5 B_{i_1}^{j_1k_1} \quad \hbox{and} \quad
  Y = \widetilde{\wedge}_5 C_{i_2}^{j_2k_2} .
  \]
  Suppose that $(u,v)$ is not in $V(J_A)$.  Then $Q(X,Y) \not= 0$. This
  implies
  \[
  Q(X,Y) \,=\, T(X,X,Y,Y) \,= \,
  T \bigl( 
  \widetilde{\wedge}_5 B_{i_1}^{j_1k_1},
  \widetilde{\wedge}_5 B_{i_1}^{j_1k_1},
  \widetilde{\wedge}_5 C_{i_2}^{j_2k_2},
  \widetilde{\wedge}_5 C_{i_2}^{j_2k_2} \bigr) \,\not= \, 0,
  \]
  and hence $(u,v) \not\in V(H_A)$. This is   a contradiction to our choice of $(u,v)$.
  
  It remains to consider the case where $v$ is not triangulable.  By
  Lemma~\ref{l:nonTriangOn}, we have $n=2$, as well as
  $v=(e_{1\leftarrow 2},e_{2\leftarrow 1})$ and $(u,v)\in V(J_A)$.
  The case where $u$ is not triangulable is symmetric, and this proves
  the theorem.
\end{proof}

The equations in Theorem~\ref{thm:six}
are fairly robust, in the sense that they work as well for 
many special position scenarios.
However, when the cameras $A_1,A_2,\ldots,A_n$
are generic then the number $9 \binom{n}{2}^2$ of octics
that cut out the divisor $V(J_A)$ inside $V_A \times V_A$
can be reduced dramatically, namely to $16$.

\begin{cor} \label{cor:onlyfew}
As a subset of the $6$-dimensional ambient space $V_A \times V_A$,
the $5$-dimensional rigid multiview variety $V(J_A)$ is cut out by
$16$ polynomials of degree class $(220..220..)$.
One choice of such polynomials is given by
\[ \quad
\begin{array}{cc}
Q\bigl(\widetilde{\wedge}_5 B^{12}_{i},\widetilde{\wedge}_5 C^{12}_{k}\bigr ), &
Q\bigl(\widetilde{\wedge}_5 B^{12}_{i},\widetilde{\wedge}_5 C^{13}_{k}\bigr )\\
Q\bigl(\widetilde{\wedge}_5 B^{13}_{i},\widetilde{\wedge}_5 C^{12}_{k}\bigr ), &
Q\bigl(\widetilde{\wedge}_5 B^{13}_{i},\widetilde{\wedge}_5 C^{13}_{k}\bigr )
\end{array} \qquad
\hbox{for all $1\leq i,k\leq 2$.}
\]
\end{cor}
\begin{proof}
  First we claim that for each triangulable point $u$ at least one of
  the matrices $B^{12}$ or $B^{13}$ has rank $5$, and the same for $v$
  with $C^{12}$ or $C^{13}$.  We prove this by contradiction. By
  symmetry between $u$ and $v$, we can assume that
  $\rk(B^{12})=\rk(B^{13})=4$. Then $u_3=e_{3\leftarrow 1}$,
  $u_2=e_{2\leftarrow 1}$, and $u_1=e_{1\leftarrow 2}=e_{1\leftarrow
    3}$.  However, this last equality of the two epipoles is a
  contradiction to the hypothesis that the focal points of the cameras
  $A_1,A_2,A_3$ are not collinear.

  Next we claim that if $B^{12}$ has rank $5$ then at least one of the
  submatrices $B^{12}_{1}$ or $B^{12}_{2}$ has rank $5$, and the same
  for $B^{13}$, $C^{12}$ and $C^{13}$.  Note that the bottom $4
  {\times} 6$ submatrix of $B^{12}$ has rank $4$, since the first four
  columns are linearly independent, by genericity of $A_1$ and $A_2$.
  The claim follows.
\end{proof}

\section{Other Constraints, More Points, and No Labels}

In this section we discuss several extensions of
our results. A first observation is that there was
nothing special about the constraint $Q$ in
(\ref{eq:biquadratic}). For instance, fix positive
integers $d$ and $e$, and let
$Q(X,Y)$ be any irreducible polynomial that is bihomogeneous of degree $(d,e)$.
Its variety $V(Q)$ is a hypersurface of degree $(d,e)$ in $\PP^3 \times \PP^3$.
The following analogue to Theorem \ref{thm:six} holds, if we define the map $\psi_A$ as in (\ref{eq:rigidMap}).

\begin{thm} \label{thm:seven} The closure of the image of the map $\psi_A$ is cut out in $V_A \times V_A $ by $\,9
  \binom{n}{2}^2 $ polynomials of degree class $(d,d,0,\ldots,e,e,0,\ldots)$. In other words, the equations coming from
  any two pairs of cameras suffice set-theoretically.
\end{thm}

\begin{proof}
The tensor $T$  that represents $Q$ now lives in
$\Sym_d(\RR^4) \otimes \Sym_e(\RR^4)$.
The polynomial  $(\ref{eq:zweizweizweizwei})$ vanishes
on the image of $\psi_A$ and has degree $(d,d,e,e)$.
The proof of Theorem~\ref{thm:six} remains valid.
The surface $X^Q$  in (\ref{eq:surface})
is irreducible of degree $e$ in $\PP^3$.
These polynomials cut out that image inside $V_A \times V_A$.
\end{proof}

\begin{rem}
In the generic case,
we can replace $9\binom{n}{2}^2$ by 16,
as in Corollary~\ref{cor:onlyfew}.
\end{rem}

Another natural generalization is to consider $m$ world points
$X_1,\ldots,X_m$ that are linked by one or several
constraints in $(\PP^3)^m$. Taking images with $n$
cameras, we obtain a variety $V(J_A)$ which lives in $(\PP^2)^{mn}$.
 For instance, if $m=4$ and $X_1,X_2,X_3,X_4$ 
are constrained to lie on a plane in $\PP^3$,
 then $\,Q = {\rm det}(X_1,X_2,X_3,X_4)\,$ and
  $V(J_A)$ is a variety of dimension $11$ in $(\PP^2)^{4n}$.
  Taking $ 6 {\times} 6$-matrices $B,C,D,E$ as in 
(\ref{eq:triangulation}) for the four points, we then form
\begin{equation}
\label{eq:fourbyfour}
\quad  \det \bigl(\widetilde{\wedge}_5 B_i,
 \widetilde{\wedge}_5 C_j,
  \widetilde{\wedge}_5 D_k,
\widetilde{\wedge}_5 E_l \bigr)\qquad \text{for all } 1\leq i,j,k,l\leq 6.
\end{equation}
For $n=2$ we verified with \macaulay that the prime ideal $J_A$ is generated by $16$ of these determinants, along with
the four bilinear forms for ${V_A}^4$.

\begin{prop}
  The variety $V(J_A)$ is cut out in ${V_A}^4 $ by the
   $\,16 \binom{n}{2}^4$ polynomials
    from~{\rm (\ref{eq:fourbyfour})}.  In other words, the equations coming
  from any two pairs of cameras suffice set-theoretically.
\end{prop}
\begin{proof}
  Each polynomial~(\ref{eq:fourbyfour}) is in $J_A$.  The proof of Theorem~\ref{thm:six} remains valid.  The planes
  $(X_i,X_j,X_k)^Q$ intersect the baseline $\beta_{12}$ in one point each.
\end{proof}
\smallskip

To continue the theme of rigidity, we may
impose distance constraints on pairs
of points. Fixing a nonzero distance $d_{ij} $
between points  $i$ and $j$ gives
\begin{small}
\[
 Q_{ij} =     (X_{i0} X_{j3} - X_{j0} X_{i3})^2 
+ (X_{i1} X_{j3} - X_{j1} X_{i3})^2 
+ (X_{i2} X_{j3} - X_{j2} X_{i3})^2 
- d_{ij}^2 X_{i3}^2 X_{j3}^2.
\]
\end{small}
We are interested in the image of the variety
$\mathcal{V} = V(Q_{ij}: 1 \leq i < j \leq m)$ 
under the multiview map $\psi_A$ that takes
 $(\PP^3)^m$  to $(\PP^2)^{mn}$.
For instance, for $m=3$, we consider the
variety $\mathcal{V} = V(Q_{12},Q_{13},Q_{23})$ in $(\PP^3)^3$,
and we seek the equations for its image 
under the multiview map $\psi_A$ into $(\PP^2)^{3n}$. 
Note that $\mathcal{V}$ has dimension $6$, unless
we are in the collinear case. Algebraically,
\begin{equation}
\label{eq:threecollinear} \,\,
(d_{12}+d_{13}+d_{23})
(d_{12}+d_{13}-d_{23})
(d_{12}-d_{13}+d_{23})
(-d_{12}+d_{13}+d_{23}) \,\,=\,\, 0.
\end{equation}
If this holds then ${\rm dim}(\mathcal{V}) = 5$.
The same argument as  in Theorem \ref{thm:six} yields:

\begin{cor}  The rigid multiview variety  $\overline{\psi_A(\mathcal{V})}$
has dimension six, unless (\ref{eq:threecollinear}) holds,
in which case the dimension is five.
It has real points if and only if $d_{12}, d_{13}, d_{23}$ satisfy the triangle inequality.  It is cut out in ${V_A}^3$ by $27\binom{n}{2}^2$ biquadratic equations,
coming from the $9 \binom{n}{2}^2$ equations for any two of the three points.
\end{cor}

In many computer vision applications, the $m$ world points and their
images in $\PP^2$ will be unlabeled. To study such questions, we
propose to work with the {\em unlabeled rigid multiview variety}. This
is the image of the rigid multiview variety under the quotient map
$\bigl((\PP^2)^m \bigr)^n \rightarrow\bigl( \Sym_m(\PP^2)\bigr)^n$.

Indeed, while labeled configurations in the plane are points in
$(\PP^2)^m $, unlabeled configurations are points in the {\em Chow
  variety} $\Sym_m(\PP^2)$.  This is the variety of ternary forms that
are products of $m$ linear forms (cf.~\cite[\S 8.6]{Lan}). It is
embedded in the space $ \PP^{\binom{m+2}{2}-1} $ of all ternary forms
of degree~$m$.

\begin{exmp}
\label{chow}
Let $m=n=2$. The Chow variety $\Sym_2(\PP^2)$
is the hypersurface in $\PP^5$ defined by the
determinant of a  symmetric $3 \times 3$-matrix $(a_{ij})$.
The quotient map $(\PP^2)^2 \rightarrow \Sym_2(\PP^2)
\subset \PP^5$
is given by the formulas
\[
\begin{matrix} 
a_{00} = 2 u_{10} v_{10}, & 
a_{11} = 2 u_{11} v_{11}, & 
a_{22} = 2 u_{12} v_{12}, \\
a_{01} = u_{11} v_{10}+u_{10} v_{11} , & 
a_{02} = u_{12} v_{10}+u_{10} v_{12}, & 
a_{12} = u_{12} v_{11}+u_{11} v_{12}. \\
\end{matrix}
\]
Similarly, for the two unlabeled images under the second camera we use 
\[
\begin{matrix} 
b_{00} = 2 u_{20} v_{20}, & 
b_{11} = 2 u_{21} v_{21}, & 
b_{22} = 2 u_{22} v_{22}, \\
b_{01} = u_{21} v_{20}+ u_{20} v_{21} , & 
b_{02} = u_{22} v_{20}+ u_{20} v_{22}, & 
b_{12} = u_{22} v_{21}+ u_{21} v_{22}. \\
\end{matrix}
\]
The unlabeled rigid multiview variety is the image of $V(J_A) \subset
V_A \times V_A$ under the quotient map that takes two copies of
$\,(\PP^2)^2 \,$ to two copies of $\Sym_2(\PP)^2 \subset \PP^5$.  This
quotient map is given by $(u_1,v_1) \mapsto a, \, (u_2,v_2) \mapsto
b$.

We first compute the image of $V_A \times V_A$
in $\PP^5 \times \PP^5$,
denoted $\Sym_2(V_A)$. Its ideal
has seven minimal generators,
three of degree $(1,1)$, and one each in
degrees $(3,0), (2,1), (1,2), (0,3)$.
The generators in degrees $(3,0)$ and $(0,3)$
are  ${\rm det}(a_{ij})$ and ${\rm det}(b_{ij})$.
The five others depend on the cameras $A_1, A_2$.

Now, to get equations for the unlabeled rigid multiview variety, we
intersect the ideal $J_A$ with the subring $\RR[a,b]$ of bisymmetric
homogeneous polynomials in $\RR[u,v]$. This results in nine new
generators which represent the distance constraint.  One of them is a
quartic of degree $(2,2)$ in $(a,b)$.  The other eight are quintics,
four of degree $(2,3)$ and four of degree $(3,2)$.
\end{exmp}

Independently of the specific constraints considered in this paper, it
is of interest to characterize the pictures of $m$ unlabeled points
using $n$ cameras.  This gives rise to the {\em unlabeled multiview
  variety} $\,\Sym_m(V_A) $ in $
{\bigl(}\PP^{\binom{m+2}{2}-1}{\bigr)}^{n} $. It would be desirable to
know the prime ideal of $\Sym_m(V_A) $ for any $n$ and~$m$.


\bigskip
\medskip

\noindent {\bf Acknowledgements.}\ This research is carried out in the framework of Matheon supported by Einstein Foundation Berlin. Further support by Deutsche
Forschungsgemeinschaft (Priority Program 1489: ``Experimental methods 
in algebra, geometry, and number theory'') and
the US National Science Foundation (grant DMS-1419018).


\begin{appendix}
\section{Computations}
\label{compu}
We performed several random experiments in this paper.
Our hardware was a cluster with Intel Xeon X2630v2 Hexa-Cores (2.8 GHz) and 64GB main memory per node.
The software was \macaulay, version 1.8.2.1 \cite{M2}.
All computations were single-threaded.

The tests were repeated several times with random input.
The exact running times vary, even with identical input; the Table~\ref{tab:gens} lists the average values.
It is not surprising that increasing $n$, the number of cameras, increases the running times considerably.
Therefore we adapted the number of experiments according to $n$.

For all the statements in Section~2 regarding two cameras, (i.e. $n=2$) we performed at least 1000 computations, with one exception.
The statement regarding the universal Gr\"obner basis in Remark~\ref{rem:octic} is based on 20 experiments.
Regarding three and four cameras (i.e., $n\in\{3,4\}$) we performed at least 100 computations each.
For $n=5$ we performed at least 20 computations each.
Example~\ref{chow} was checked with 50 choices of random cameras.

In Listing~\ref{algo1} we show \macaulay code which can be employed to establish Proposition~\ref{prop:n2}.
The complete code for all our results can be accessed via \url{http://www3.math.tu-berlin.de/combi/dmg/data/rigidMulti/}.

Lines 1--4 define the rings in which the computations take place.
Lines 6--7 produce random camera matrices.
Here the code shown differs slightly from the code used.
What we omitted is the extra code which checks the matrices before they are processed.
To assert general position we check that none of the minors vanish as in \cite[\S2]{AST}.
However, our experiments suggest that it suffices to check that the focal points of the cameras are in linear general position.
The multiview map $\phi_A$ from (\ref{eq:multiMap}) is encoded in lines 11--14.
Line 15 is the rigid constraint (\ref{eq:biquadratic}).
The actual computation is the elimination in line~16.
The rigid multiview ideal $J_A$ is defined in lines 17--18, and the final output are the multidegrees of $J_A$.

\begin{lstlisting}[label=algo1, caption=Compute $J_A$ for two cameras]
R1 = QQ[u_(1,0)..u_(1,2)] ** QQ[u_(2,0)..u_(2,2)] **
     QQ[v_(1,0)..v_(1,2)] ** QQ[v_(2,0)..v_(2,2)];
R2 = QQ[X_0..X_2] ** QQ[Y_0..Y_2]; 
S  = R1 ** R2;

n  = 2;
for i from 1 to n do (A_i = random(ZZ^3,ZZ^4,Height=>20););

I = ideal();
for j from 1 to n do (
   I = I + minors(2,A_j * (genericMatrix(S,X_0,3,1)||matrix{{1}})|
       genericMatrix(S,u_(j,0),3,1));
   I = I + minors(2,A_j * (genericMatrix(S,Y_0,3,1)||matrix{{1}})|
       genericMatrix(S,v_(j,0),3,1)); );
I = I + ideal((X_0-Y_0)^2 + (X_1-Y_1)^2 + (X_2-Y_2)^2-1);
I = eliminate({X_0,X_1,X_2,Y_0,Y_1,Y_2},I);

F = map(R1,S);
J = F(I);

degrees(J)
\end{lstlisting}

In these computations the world coordinates are dehomogenized by setting the last coordinate to $1$, as explained at the end of Section~\ref{sec:234}.
Notice that the code below line~7 does not need to be modified if we increase $n$.
\end{appendix}

\bigskip \bigskip \bigskip

\footnotesize 

\noindent \textbf{Authors' addresses:}

\noindent
Michael Joswig, Technische Universit\"at Berlin,  Germany,
\texttt{joswig@math.tu-berlin.de}

\noindent Joe Kileel,  University of California, Berkeley, USA,
\texttt{jkileel@math.berkeley.edu}

\noindent Bernd Sturmfels,  University of California, Berkeley, USA,
\texttt{bernd@math.berkeley.edu}

\noindent
Andr\'e Wagner, Technische Universit\"at Berlin,
Germany, \texttt{wagner@math.tu-berlin.de}

\end{document}